\documentclass[11pt,a4paper,twoside]{article}
\usepackage[utf8]{inputenc}

\usepackage[a4paper, total={6.2in, 8.2in}]{geometry}

\usepackage{xcolor}
\usepackage{amsmath,amssymb,amsthm,amsfonts}
\usepackage{tikz}

\usepackage{bbm}
\usepackage{mathtools, nccmath, relsize}
\usepackage{hyperref}
\usepackage{authblk}

\theoremstyle{definition}
\newtheorem{theorem}{Theorem}[section] 
\newtheorem{definition}[theorem]{Definition}

\newtheorem{proposition}[theorem]{Proposition}

\title{Why the 1-Wasserstein distance is the area between the two marginal CDFs}
\author{Marco De Angelis}
\author{Ander Gray}
\affil{Institute for Risk and Uncertainty, University of Liverpool}

\begin{document}

\maketitle

\begin{abstract}
We elucidate why the 1-Wasserstein distance $W_1$ coincides with the area between the two marginal cumulative distribution functions (CDFs). We first describe the Wasserstein distance in terms of copulas, and then show that $W_1$ with the Euclidean distance is attained with the $M$ copula. Two random variables whose dependence is given by the $M$ copula manifest perfect (positive) dependence. If we express the random variables in terms of their CDFs, it is intuitive to see that the distance between two such random variables coincides with the area between the two CDFs.
\end{abstract}

\vspace{1em}

\section{The 1-Wasserstein distance in terms of copulas}
The Wasserstein distance is a popular metric often used to calculate the distance between two probability measures. It is  a metric because it obeys the four axioms: (1) identity of indiscernibles, (2) symmetry,  and (3) triangle inequality (4) non-negativity.  The general formal definition of the metric is attributed to Kantorovich and Wasserstein, among many more authors \footnote{See bibliographical extract at Section \ref{sec:biblionote} at the end of this document.}:

\begin{definition}[Kantorovich-Wasserstein metric]\label{def:kw_metric_general}
Let $(\mathcal{X} , d)$ be a Polish metric space, and let $p \in [1, \infty)$. For any two marginal measures $\mu$ and $\nu$ on $\mathcal{X}$, the Kantorovich-Wasserstein distance of order $p$ between $\mu$ and $\nu$ is given by
\begin{align*}
W_p(\mu, \nu) & = \left( \inf_{\pi \in \Pi(\mu, \nu)}   \int_{\mathcal{X} } d(x,y)^p \ d\pi(x,y) \right)^{1/p} \\[1ex]
						 &  = \ \inf \left\{ \mathbb{E} \left[ d(X, Y)^p  \right]^{1/p} ,\ \mu = F_{\mu}(X),\  \nu = F_{\nu}(Y) \right\},
\end{align*}
where, $\Pi(\mu, \nu)$ denotes the collection of all measures on $\mathcal{X}$ with marginals $\mu$ and $\nu$. The set $\Pi(\mu, \nu)$ is also called the set of all \emph{couplings} of $\mu$ and $\nu$.
\end{definition} 

\noindent The above definition comes from optimal transport theory \cite{villani2009optimal}, where couplings $\Pi$ denote a transport plan for moving, from $x$ to $y$, the (probability) mass of a pile of soil distributed as $\mu$  to a pile distributed as $\nu$, and given that the distance $d$ is the cost of moving a unit of mass from $x$ to $y$. Optimal transport theory entails finding the optimal coupling that minimises the overall work.

Definition \ref{def:kw_metric_general} is very general, so we specialise this definition to the case of two random variables on the real line $\mathbb{R}$, for the Euclidean distance $d(x,y) \equiv |x-y|$, and the case of degree one, $p=1$. We also change the notation to a more standard notation of probability theory:

\begin{definition}[Wasserstein distance]\label{def:kw_metric_realline}
Let $\mathcal{X} = \mathbb{R} \times \mathbb{R}$. For any two random variables $X$ and $Y$, with distribution functions $F_{X}$ and $F_{Y}$, the \emph{Wasserstein distance} between $X$ and $Y$ is given by

\begin{align*}
W_1(X, Y) & = \inf_{F_{XY} \in \mathcal{F}}   \int_{\mathcal{X} } |x -y| \ dF_{XY}(x,y)  \\[1ex]
		&  = \ \inf_{F_{XY} \in  \mathcal{F}}   \left\{ \mathbb{E}_{F_{XY}}|X -Y|  \right\},
\end{align*}

\noindent where $\mathcal{F}$ denotes the collection of all joint distributions on $\mathcal{X}$ with marginal distributions $F_{X}$ and $F_{Y}$, and $\mathbb{E}_{F_{XY}}$ is the expectation given that the joint distribution of $X$ and $Y$ is $F_{XY}$.
\end{definition} 

Since the above metric involves searching through a collection of joint distributions with fixed marginals, it is possible to express the joint distribution $F_{XY}$ in terms of the marginals $F_{X}$, $F_{Y}$ and copula  $C$ using Sklar's theorem \cite{schweizer2011probabilistic}: $F_{XY}(x,y) = C(F_{X}(x), F_{Y}(y))$. Let $\mathcal{C}$ be the set of all bivariate copulas (2-copulas). Then this definition follows:

\begin{definition}[Wasserstein distance with copulas] \label{def:w_metric_copulas}
Let $\mathcal{X} = \mathbb{R} \times \mathbb{R}$. Let $F_{X}$ and $F_{Y}$ be two marginal distributions on $\mathcal{X}$, and $C(F_{X}(x), F_{Y}(y))$ be their joint cumulative distribution in terms of copula. Then the \emph{Wasserstein distance} between $F_{X}$ and $F_{Y}$ is given by 
\begin{align*}
W_1(X, Y) & =  \inf_{C \in \mathcal{C}}   \int_{ \mathcal{X} } |x-y| \ dC(F_{X}(x), F_{Y}(y))  \\[1ex]
						 &  = \ \inf_{C \in  \mathcal{C}}   \left\{ \mathbb{E}_C |X-Y|   \right\},
\end{align*}

\noindent where $\mathcal{C}$ denotes the collection of all 2-copulas, and $\mathbb{E}_{C}$ is the expectation given that the copula between $X$ and $Y$ is $C$.

\end{definition}

Using Definition \ref{def:w_metric_copulas}, $W_1$ can be re-written in terms of the \emph{generalised inverses} \footnote{\url{https://en.wikipedia.org/wiki/Cumulative_distribution_function\#Inverse_distribution_function_(quantile_function)}}. Given that $u = F_{X}(x)$, and $v = F_{Y}(y)$, and so $x = F^{-1}_{X}(u)$ and $y = F^{-1}_{Y}(v)$, the integration may be performed on the unit square $[0,1]^2$

\begin{align} \label{eq:wasser_pseudoinv_copula}
   \nonumber  W_{1}(X, Y) &= \inf_{C \in \mathcal{C}}  \int_{\mathcal{X} } |x-y|  \ d C(F_{X}(x),F_{Y}(y)) \\
    						& = \inf_{C \in \mathcal{C}}  \int_{[0,1]^2 } |F^{-1}_{X}(u)-F^{-1}_{Y}(v)|  \ dC(u,v)  .
\end{align}

\section{The optimal distance $W_1$ holds for the case of perfect dependence between $X$ and $Y$, i.e. for $C=M$}

With Definition \ref{def:w_metric_copulas} in terms of copulas, an exact solution to the infimum in \eqref{eq:wasser_pseudoinv_copula} can be obtained by substituting $C$ with the $M$ copula. We are ready to state the main result in the following theorem.

\begin{theorem}\label{th:main}
\textit{
The infimum in \eqref{eq:wasser_pseudoinv_copula} over all 2-copulas $\mathcal{C}$ is attained at the $M$ copula, that is $C(u,v)=M(u,v)=\min\{u,v\}$, which is equivalent to demand that $u=v$,  leading to }

\begin{equation}\label{eq:main_integral}
W_{1}(X, Y) =  \int_{[0,1] } |F^{-1}_{X}(u)-F^{-1}_{Y}(u)|  \ du .
\end{equation}

\end{theorem}

\begin{proof}

Let $X, Y \in \mathbb{R}$, we want to show that \eqref{eq:wasser_pseudoinv_copula}  has exact solution for $C = M$.  From Definition \ref{def:w_metric_copulas} the following holds:

\begin{equation} \label{eq:expectation}
\mathbb{E}_C |X-Y| =  \int_{[0,1]^2 } |F^{-1}_{X}(u)-F^{-1}_{Y}(v)|  \ dC(u,v)  .
\end{equation}

So we study the expectation \eqref{eq:expectation}. In Vallender  \cite{vallender1974calculation}, a formula is provided to express \eqref{eq:expectation} in terms of the joint probability distribution as follows:

\begin{equation*}\label{eq:expectation_1}
\begin{matrix} 
\mathbb{E} |X-Y| & =  & \displaystyle  \int_{-\infty}^{\infty} \left( P(X<t, Y \geq t) + P(X \geq t, Y<t) \right)\ dt \\[12pt]

  & = &  \displaystyle \int_{-\infty}^{\infty} \left( P(X<t) + P(Y<t) - 2 P(X < t,\ Y < t) \right) \ dt .
\end{matrix}
\end{equation*}

Because $P(X < t, Y < t)$ is the joint cumulative distribution, using Sklar's theorem we have $P(X < t, Y < t) = C \left( F_{X}(t), F_{Y}(t) \right)$. Then we can re-write the expectation in terms of the distribution functions as follows:

\begin{equation}\label{eq:expectation_with_C}
\begin{matrix} 
\mathbb{E}_C |X-Y| = &  \displaystyle \int_{-\infty}^{\infty} \left( F_{X}(t) + F_{Y}(t) - 2 C \left( F_{X}(t), F_{Y}(t) \right) \right) \ dt.
\end{matrix}
\end{equation}

\noindent All 2-copulas are bounded above and below by two copulas $W(u,v) = \max\{u+v-1,0\}$ and $M(u,v) = \min\{u,v\}$: 
$$ W(u,v) \leq C \left( u, v \right) \leq M(  u, v),$$

\noindent
thus the integrand in \eqref{eq:expectation_with_C} has the following lower bound for $C \in \mathcal{C}$:

\begin{equation}\label{eq:integrand_inequality}
    F_{X}(t) + F_{Y}(t) - 2 M \left( F_{X}(t), F_{Y}(t) \right) \leq  F_{X}(t) + F_{Y}(t) - 2 C \left( F_{X}(t), F_{Y}(t) \right).
\end{equation}

\noindent
In Figure \ref{fig:wizard_hat} we show with an example involving two random variables, that \eqref{eq:integrand_inequality} holds for some Gaussian copulas. Substituting the integrand of \eqref{eq:expectation_with_C} with the left hand side of \eqref{eq:integrand_inequality} we have that:

\begin{equation*}\label{eq:M_lower_expectation}
\displaystyle \int_{-\infty}^{\infty} \left( F_{X}(t) + F_{Y}(t) - 2 M \left( F_{X}(t), F_{Y}(t) \right) \right) \ dt \ \leq  \  \mathbb{E}_C |X-Y|.
\end{equation*}

The expectation over an arbitrary copula $\mathbb{E}_{C} |X-Y|$ has a minimum for $C=M$, thus the following holds for all copulas $C \in \mathcal{C}$:

\begin{equation*}\label{eq:w1_M_copula}
\int_{[0,1]^2 } |F^{-1}_{X}(u)-F^{-1}_{Y}(v)|  \ dM(u,v)   \ \leq \ \int_{[0,1]^2 } |F^{-1}_{X}(u)-F^{-1}_{Y}(v)|  \ dC(u,v).
\end{equation*}

Finally, because $M(u,v)=\min \{u,v \}$ is equivalent to demand perfect positive dependence $u = v$, we have that:

\begin{equation*}\label{eq:w1_M_copula_2}
\int_{[0,1]^2 } |F^{-1}_{X}(u)-F^{-1}_{Y}(v)|  \ dM(u,v)   = \int_{[0,1] } |F^{-1}_{X}(u)-F^{-1}_{Y}(u)|  \ du,
\end{equation*}

\noindent
which concludes the proof. 
\end{proof}

\begin{figure}\label{fig:wizard_hat}
\textbf{\boldmath$F_{X}(t) + F_{Y}(t) - 2C(F_{X}(t),F_{Y}(t))$} 
    \centering
    
    \includegraphics[width=0.67\textwidth]{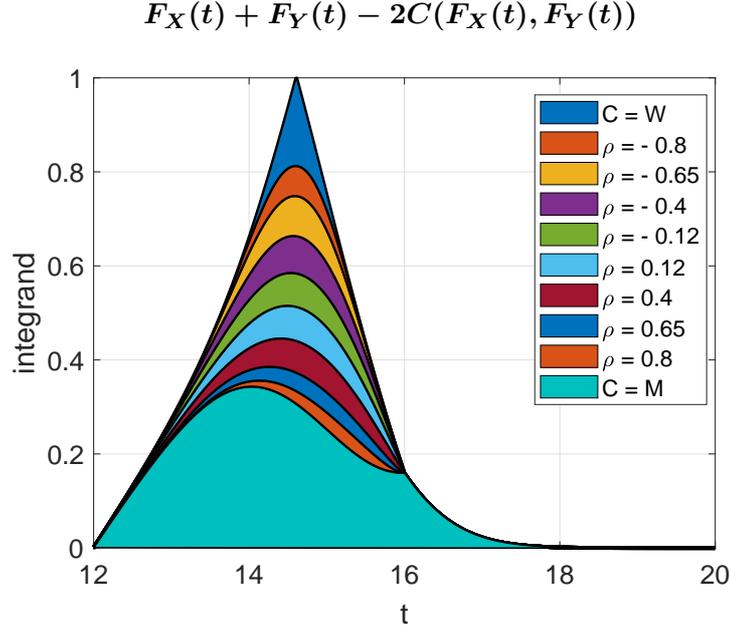}
    \caption{Integrand of \eqref{eq:expectation_with_C} for $X \sim N(15,1)$ and $Y \sim U(12,16)$, with different copulas. Gaussian copulas with parameters $\rho = \{-1, -0.8,-0.64,-0.4,-0.12,0.64,0.4,0.12,0.8, 1\}$ were used. Note that $\rho = -1$ corresponds to the $W$ copula, and $\rho = 1$ is $M$ copula. The area underneath these curves is the integral in \eqref{eq:expectation_with_C}. The smallest area corresponds to the case of the $M$ copula, whilst the largest area corresponds to the case of the $W$ copula.}
\end{figure}

\textbf{Remark}: In the above proof, two ways of solving for $\mathbb{E}_{C}[|X-Y|]$ were given: the first in terms of the distribution functions: 

\begin{equation*}
    \mathbb{E}_{C}|X-Y| = \displaystyle \int_{-\infty}^{\infty} \left( F_{X}(t) + F_{Y}(t) - 2 C \left( F_{X}(t), F_{Y}(t) \right) \right) \ dt,
\end{equation*}

\noindent and the second in terms of the inverses using a Lebesgue integral and the copula $C$:

\begin{equation*}
    \mathbb{E}_{C}|X-Y| = \int_{[0,1]^2 } |F^{-1}_{X}(u)-F^{-1}_{Y}(v)|  \ dC(u,v).
\end{equation*}

\section{$W_1$ for the case of stochastic dominance}

\begin{definition}[Stochastic dominance]
Let $X, Y \in \mathbb{R}$ be two random variables with distribution functions $F_X$ and $F_Y$, and corresponding inverses $F_X^{-1}$ and $F_Y^{-1}$. Then we say that $X$ \emph{dominates} $Y$ if and only if 
$$F_X^{-1}(u) > F_Y^{-1}(u), \ \ \ \text{for}  \ u \in [0,1].$$
This will be denoted by $X \succ Y$.
\end{definition}

\begin{proposition}[$W_1$ under dominance]\label{prop:dominance}
\textit{
Let $X, Y \in \mathbb{R}$ be two random variables, with $X \succ Y$. Then we have that}

$$W_{1}(X, Y) = \mathbb{E} X - \mathbb{E} Y.$$
\end{proposition}

\begin{proof}
From Theorem \ref{th:main}, we know that \eqref{eq:main_integral} holds, and by hypothesis $X \succ Y$, then

\begin{align*}
    W_{1}(X, Y)  & =  \int_{[0,1] } |F^{-1}_{X}(u)-F^{-1}_{Y}(u)|  \ du \\
                & = \int_{[0,1] } F^{-1}_{X}(u)-F^{-1}_{Y}(u)  \ du \\
                & = \int_{[0,1] } F^{-1}_{X}(u) \ du - \int_{[0,1] }F^{-1}_{Y}(u)  \ du \\
                 & = \int_{-\infty}^{\infty} x \ dF_{X}(x) - \int_{-\infty}^{\infty} y \ dF_{Y}(y)\\
                & = \mathbb{E} X - \mathbb{E} Y.
\end{align*}

\end{proof}

\section{Non-overlapping ranges}

\begin{definition}[Finite support]
Let $X$ be a random variable. We say that $X$ has \emph{finite support} if $X \in [\underline{x}, \overline{x}]$, such that $-\infty < \underline{x} \leq \overline{x}< \infty$, 
\begin{equation*}
\underline{x} = \sup_{\mathbb{R}} \{x : F_{X}(x)=0 \}, \ \ \ \  \overline{x} = \inf_{\mathbb{R}} \{x : F_{Y}(x)=1 \}.
\end{equation*}
\end{definition}

\begin{proposition}[Non-overlapping ranges]\label{prop:no_overlap}
\textit{
Let $X, Y \in \mathbb{R}$ be two random variables with finite support, whose ranges do not overlap: $\overline{x} < \underline{y} \ \  \text{or} \ \ \overline{y} < \underline{x}$. Then the expectation $\mathbb{E}_{C} |X-Y| $ is a singleton, whose only element is 
}

\begin{equation}
    \left\{\begin{matrix}
    \mathbb{E} X - \mathbb{E} Y, \ \ \ \text{if} \ \ \overline{y} < \underline{x}
     \\ 
     \mathbb{E} Y - \mathbb{E} X, \ \ \ \text{if} \ \ \overline{x} < \underline{y}.
    \end{matrix}\right.
\end{equation}

\begin{proof}
Without loss of generality, assume $\overline{y} < \underline{x}$. Thus there is no value of $X$ smaller than any value of $Y$. Then the absolute value is
$|X-Y| = X-Y, \  X \in [\underline{x}, \overline{x}], \ Y \in [\underline{y}, \overline{y}]$. Therefore 

$$\mathbb{E} |X-Y| = \mathbb{E} \ [X-Y] = \mathbb{E} X - \mathbb{E} Y.$$
The counter-argument applies to the case $\overline{x} < \underline{y}$, where $\mathbb{E} |X-Y| = \mathbb{E} Y - \mathbb{E} X.$
\end{proof}

\end{proposition}

Note that Propositions \ref{prop:dominance} and \ref{prop:no_overlap}  are very useful for both theoretical and computational reasons.
From Proposition \ref{prop:dominance} follows that the $W_1$ distance between two random variables under dominance is given by the difference of their expected values; whilst from Proposition \ref{prop:no_overlap} follows that the distance between two bounded random variables, whose ranges do not overlap is simply the difference of their expected values, regardless of their dependence. Moreover, the computation of such distance need not evaluate the integral \eqref{eq:main_integral}, thus can be computed very quickly.

\section{Extract of bibliographical note from Villani  \cite{villani2009optimal}}\label{sec:biblionote}

``The terminology of Wasserstein distance (apparently introduced by Dobrushin) is very questionable, since (a) these distances were discovered and rediscovered by several authors throughout the twentieth century, including (in chronological order) Gini [417, 418], Kantorovich [501], Wasserstein [803], Mallows [589] and Tanaka [776] (other contributors being Salvemini, Dall'Aglio, Hoeffding, Fr{\'e}chet, Rubinstein, Ornstein,  so in particular and maybe others); (b) the explicit definition of this distance is not so easy to find in Wasserstein's work; and (c) Wasserstein was only interested in the case p = 1. By the way, also the spelling of Wasserstein is doubtful: the original spelling was Vasershtein. (Similarly, Rubinstein was spelled Rubinshtein.) These issues are discussed in a historical note by Ru{\"s}chendorf [720], who advocates the denomination of \textit{minimal Lp-metric} instead of \textit{Wasserstein distance}. Also Vershik [808] tells about the discovery of the metric by Kantorovich and stands up in favor of the terminology \textit{Kantorovich distance}''.  
For the references whose number is displayed in the above extract the reader is referred to Villani \cite{villani2009optimal}.

\bibliographystyle{unsrt}
\bibliography{biblio}

\begin{thebibliography}{1}

\bibitem{villani2009optimal}
C{\'e}dric Villani.
\newblock {\em Optimal transport: old and new}, volume 338.
\newblock Springer, 2009.

\bibitem{schweizer2011probabilistic}
B.~Schweizer and A.~Sklar.
\newblock {\em Probabilistic {Metric} {Spaces}}.
\newblock Dover {Books} on {Mathematics}. Dover Publications, 2011.

\bibitem{vallender1974calculation}
SS~Vallender.
\newblock Calculation of the wasserstein distance between probability
  distributions on the line.
\newblock {\em Theory of Probability \& Its Applications}, 18(4):784--786,
  1974.

\end{thebibliography}

\end{document}